\documentclass{article}

\usepackage{arxiv}

\usepackage[utf8]{inputenc} 
\usepackage[T1]{fontenc}    
\usepackage{hyperref}       
\usepackage{url}            
\usepackage{booktabs}       
\usepackage{amsfonts}       
\usepackage{nicefrac}       
\usepackage{microtype}      
\usepackage{lipsum}	
\usepackage{amsmath}
\usepackage{amsthm}
\newtheorem{definition}{Definition}
\newtheorem{remark}{Remark}
\newtheorem{theorem}{Theorem}
\newtheorem{lemma}{Lemma}
\newtheorem{proposition}{Proposition}

\title{Degenerate elliptic problem with a singular nonlinearity}

\author{
Abdelaaziz Sbai and Youssef El hadfi\\
Laboratory LIPIM\\
National School of Applied Sciences Khouribga\\
Sultan Moulay Slimane University, Morocco\\
  \texttt{yelhadfi@gmail.com}\\
    \texttt{sbaiabdlaaziz@gmail.com}\\
}

\begin{document}
\maketitle
\begin{abstract}
In this paper, we prove existence and regularity results for solutions of some nonlinear Dirichlet problems for an elliptic  equation defined by a degenerate coercive operator and a singular right hand side. 
\begin{equation}\label{01}
\left\{
\begin{array}{lll}
-\displaystyle\mbox{div}( a(x,u,\nabla u))&=\displaystyle\frac{f}{u^{\gamma}} & \mbox{ in } \Omega \\
u&>0 &\mbox{ in }\Omega \\
u&=0  &\mbox{ on }  \delta\Omega
\end{array}
\right.
\end{equation}
where $\Omega $ is bounded open subset of $I\!\!R^{N}(N\geq2),$ $\gamma>0$ and $ f$ is a nonnegative function that belongs to some Lebesgue space.
\end{abstract}


\section{Introduction}
The aim of this work is the study of the following boundary value problem 
\begin{equation}\label{11}
\left\{
\begin{array}{lll}
&-\mbox{div}( a(x,u,\nabla u)) =fh(u)& \mbox{ in } \Omega \\
&u>0 &\mbox{ in }\Omega \\
&u=0  &\mbox{ on }  \delta\Omega
\end{array}
\right.
\end{equation}
with $\Omega$ a bounded open subset of $I\!\!R^{N}$, $ N\geq 2, $\,\, $N>p>1$, $f $ is non negative and it belongs to 
$L^{m}(\Omega)$ for some $m\geq1$. Finally the singular sourcing
$h:(0,\infty)\longrightarrow (0,\infty)$ is continuous and bounded, such that the following properties hold true 
$\exists C, \gamma >0   \mbox{ s.t } h(s)\leq\frac{C}{s^{\gamma}} \,\forall s\in (0,+\infty ).$
\par Let us give the precise assumptions on the problems that we will
study. Let $\Omega$ be a bounded open subset of $ I\!\!R^{N}$,$N\geq2,$ let $N>p>1$ and let $a:\Omega\times I\!\!R\times I\!\!R^{N}\longrightarrow 
I\!\!R^{N}$ be carath\'eodory function (that is $a(.,t,\xi)$is measurable on $\Omega$ for every $(t,\xi) $ in $I\!\!R\times I\!\!R^{N}$ and $a(x,.,.)$ is continuous on $I\!\!R\times I\!\!R^{N}$ for almost every $x$ in $\Omega$), such that the flowing assumptions hold :
\begin{equation}\label{12}
a(x,t,\xi).\xi\geq b(|t|)|\xi|^{p},
\end{equation} 
for almost every $x$ in $\Omega$ and for every $(t,\xi))$ in 
$I\!\!R\times I\!\!R^{N}$, where $b:I\!\!R^{+}\longrightarrow ]0,\infty[$ is a decreasing continuous such that its primitive 
\begin{equation}\label{13}
\mathcal{B}(t)=\int_{0}^{t} b(s)^{\frac{1}{p-1}}ds,
\end{equation} 
is unbounded, for the sake of simplicity, we take in \eqref{12}
\begin{equation}\label{112}
b(t)=\frac{\alpha}{(1+t)^{\theta(p-1)}},
\end{equation}
for some real number $0\leq \theta\leq 1$  and some $\alpha>0.$
\begin{equation}\label{14}
|a(x,t,\xi)|\leq \beta
\big[ a_{0}(x)+|t|^{p-1}+|\xi|^{p-1}\big],  
\end{equation} 
for almost every $x$ in $\Omega$, for every $(t,\xi) $ in $I\!\!R\times I\!\!R^{N}$, where $a_{0}$ is 
non-negative function in $L^{p^{\prime}}(\Omega),$ with $\frac{1}{p}+\frac{1}{p^{\prime}}=1$ and $ \beta\geq\alpha,$ 
\begin{equation} \label{15}
\big[a(x,t,\xi)-a(x,t,\xi^{'})\big] (\xi-\xi^{'})> 0,
\end{equation}
for almost every $x$ in $\Omega$ and for every  $t$ in $I\!\!R$, for every $\xi$,$\xi^{'}$ in $I\!\!R^{N}$, with 
$\xi \neq \xi^{'}$ 
we will then define, for $u$ in $W_{0}^{1,p}(\Omega)$ the 
non linear elliptic operator $$ A(u)=- div (a(x,u,\nabla u)).$$
In the case of linear elliptic opertors a rich amount of research has been  conducted to prove the existence  of a solution to the problem
\begin{equation*}
\left\{
\begin{array}{lll}
&-\Delta u =f h(u) &\mbox{  in } \Omega \\
&u=0 &\mbox{ on } \delta\Omega
\end{array}
\right.
\end{equation*}
An existence result pertaining to the case $h(s)=\frac{1}{s^{\gamma}}$ with $f$ being bounded away from the origin and sufficiently regular possesses a unique solution by desingularizing the  problem and then applying the sub and the super solution method. Few generalizations to this result can be found in \cite{01}. A weaker
condition on the function $f$ from $L^{m}(\Omega)$, for $m\geq 1$. In a study due to Boccardo at. \cite{03}, they have proved  the existence and uniqueness  of solutions to the problem 
\begin{equation}\label{16}
\left\{
\begin{array}{lll}
&-\Delta u =\frac{f}{u^{\gamma}}&\mbox{ in }\Omega \\
&u>0 &\mbox{ in }\Omega \\
&u=0 &\mbox{ on } \delta\Omega.
\end{array}
\right.
\end{equation}
In \cite{071, 232, 233} a nonlinear version of the above problem was studied, considering an operator as the p-laplacian $-\mbox{div} (|\nabla u|^{p-2}\nabla u)$ instead of $-\Delta u,$ the authors prove existence of and regularity results if $f$ belong to $L^{m}(\Omega).$

However, the authors analyse the following singular problem 
\begin{equation}\label{17}
\left\{
\begin{array}{lll}
&C(u)=-div(A(x) \nabla u) =g(x,u)& \mbox{ in }\Omega \\
&u=0 &\mbox{ on }\delta\Omega
\end{array}
\right.
\end{equation}
with $g(x,s)$ singular for $s=0,$ have been studied by various authors in the past.We refer in particular to the paper by Crandall at. (see \cite{05}) and to the one by Lazer at. (see \cite{01}). In these last works, the authors have dealt with the case $g(x,s)=\frac{f(x)}{u^{\gamma}},$ assuming that $f$ is a continuous function. They  proved existence and regularity results for the solutions, using the method of subsolutions and supersolutions by means of a suitable power of the first  eigenfunctions of the Laplacian in $\Omega$. The linear case, with $g(x,u)=\frac{f(x)}{u^{\gamma}},$ was then exhaustively studied by Boccardo at.  (see\cite{03}), we refer to \cite{07} also to the references to previous works in which problems of the form \eqref{17} have been examined. 

\par The last works that were done in the presence of singular term,studied the existence and regularity when the operator elliptic is linear or $p$-Laplacian does not depend on $u.$ In our work we use the elliptic nonlinear operator depends on $u$ with the degenerate  coercivity, the difficulty we face is to proved the existence, so that the standard Leary-Lions Theorem cannot be applied. To overcome this problem, we need to approximate the problem \eqref{11}, a suitable way and using Shaulder's fixed point theorem in order to prove the existence.
Of course, once one approximates the equation, both a priori
estimates  and asymptotic behavior (strong convergence in Sobolev spaces)
of the sequence of approximating solutions have to be proven in order to pass to the limit.
\begin{definition} 
	Let $f$ be in $L^{m}(\Omega)$,$m \geq 1 $. A measurable  function $u$ is a solution of \eqref{11} in the sense of distributions if  $u> 0$ a.e in $\Omega $, $ fh(u) \in L^{1}(\Omega)$ and if 
	\begin{equation}\label{19}
	\int_{\Omega} a(x,u,\nabla u) \nabla \varphi dx =\int_{\Omega}
	f h(u) \varphi dx, \mbox{ for every }   \varphi\in L^{\infty}(\Omega) \cap W_{0}^{1,p}(\Omega).
	\end{equation}
\end{definition}
Our first result is the following:
\begin{theorem} \label{Athm1}
	Let  $f\in L^{m}(\Omega)$ with $m>N\slash p,$
	assume that \eqref{12}, \eqref{112}, \eqref{14} and \eqref{15} hold true then, there  exists a function  $u \in W_{0}^{1,p}(\Omega)\cap L^{\infty}(\Omega)$ solution of \eqref{11}.
\end{theorem}
\begin{theorem}\label{Athm2} 
	Assume that \eqref{12}, \eqref{112},\eqref{14},\eqref{15} and $0< \gamma < \theta(p-1)+1 $  hold true. Let $f\in L^{m}(\Omega)$ with 
	\begin{equation}\label{113}
	m_{1}=\left(\frac{p^{\ast}}{\theta(p-1)+1-\gamma}\right)^{\prime}=\frac{Np}{Np-(N-p)[\theta(p-1)+1-\gamma]}\leq m <N\slash p,
	\end{equation}
	then, there exists at least one solutions  $u$ in $W_{0}^{1,p}(\Omega)\cap L^{r}(\Omega)$ of \eqref{11}
	\begin{equation}\label{114}
	r=\frac{Nm[(p-1)(1-\theta)+\gamma]}{N-pm}.
	\end{equation}
\end{theorem}
\begin{remark}
	If $0< \gamma < \theta(p-1)+1,$ we explicitly note that $m=m_{1} \Longleftrightarrow r=p^{\ast},$ and If $\theta=1, \gamma\rightarrow 0,$ then $m_{1}\rightarrow N\slash p,$ in this
	case. Observe that, for every $0\leq\theta\leq 1,$ we have $m_{1}\geq (p^{\ast})^{\prime} \Rightarrow f \in W^{-1,p^{\prime}}(\Omega),$ it is classical to expect a $ W^{1,p}_{0}(\Omega)$ solution.
\end{remark}
\begin{theorem} \label{Athm3}
	Assume that \eqref{12}, \eqref{112},\eqref{14},\eqref{15} and $0< \gamma < \theta(p-1)+1 $  hold true. Let $f\in L^{m}(\Omega)$ with 
	\begin{equation}\label{115}
	\displaystyle\mbox{max}(1,\frac{N}{(p-1)[N(1-\theta)+\theta]+1+\gamma(N-1)})\leq m <m_{1},
	\end{equation}
	then, there exists at least one solutions  $u$ in $W_{0}^{1,\sigma}(\Omega)\cap L^{r}(\Omega)$, that is 
	\begin{equation}\label{116}
	\sigma=\frac{Nm[(p-1)(1-\theta)+\gamma)]}{N-m((p-1)\theta+1-\gamma)}.
	\end{equation}
	and 
	\begin{equation}\label{114'}
	r=\frac{Nm[(p-1)(1-\theta)+\gamma]}{N-pm}.
	\end{equation}
\end{theorem}
\begin{remark}
	If $\gamma\rightarrow0$, the result of Theorem \ref{Athm2}, Theorem \ref{Athm3} coincides with regularity results for elliptic equation with coercivity (see(\cite{02},Theorem 1.3 and Theorem 1.7)).
\end{remark}
\begin{remark}
	If $0< \gamma<\theta (p-1)+1,$ under some condition on $f$, the summability of the solution to \eqref{01} is better than or equal to that of solution to \eqref{11} in (\cite{02},Theorem 1.7 and Theorem 1.9), since $\sigma > q $ and $ r>s $ (see \cite{02}).
\end{remark}
\begin{theorem}\label{Athm4}\qquad
	\begin{itemize}	
		\item[i)]	Let us consider $\gamma =\theta(p-1)+1$ and $f\in L^{1}(\Omega)$ the solution $u$ to \eqref{19}are uniformly bounded in 
		$W_{0}^{1,p}(\Omega)$. 
		\item[ii)] Let $\gamma > \theta(p-1)+1$ and $f\in L^{1}(\Omega)$ 
		then there exists $u$ uniformly bounded in $
		L^{\frac{\gamma+(p-1)(1-\theta)}{p}p^{*}}(\Omega)$
		to \eqref{11} in the sense of \eqref{19}, such that 
		$ u^{\frac{\gamma+(p-1)(1-\theta)}{p}}\in W_{0}^{1,p}(\Omega)$. 
	\end{itemize}
\end{theorem}
The paper is organized as follows: in the next section we will recall some definitions and proprieties of rearrangements that will play a role in our proofs, in the third section we will give a priori estimates for solutions of approximate equation, while the fourth section will be devoted to the proof of the results.
\section{Rearrangements and related properties}
In this section we recall a few notions about rearrangements. Let $\Omega$ be an open bounded set of $I\!\!R^{N}.$
If $u$ is a measurable function in $\Omega$, we define the distribution
function $\mu$ of $u$ as follows 
$$\mu_{u}(t)=\arrowvert \left\lbrace x\in \Omega :|u(x)|> t\right\rbrace  \arrowvert,\,\,\, t\geq 0.$$
Where $|E|$ denotes the Lebesgue measure of a measurable subset $E$ of $I\!\!R^{N}$. The function $\mu$ is decreasing and right-continuous. The decreasing rearrangement of $u$ is defined by 
$$u^{*}(s)=\inf\left( t\geq0 :\mu_{u}(t)\leq s\right)   \,\,\,\mbox{for}\,\,s\in [0,|\Omega|].$$
Recall that the following inequality 
$$ u^{*}(\mu_{u}(t))\leq t,$$
holds for every $t>0$ (see  \cite{06},\cite{21}). We also have (see \cite{23})
$$u^{*}(0)=ess \,\,sup|u|.$$ 
If $f$ is any continuous increasing map from  $[0,\infty]$ into
$[0,\infty]$  such that $f(0)=0$, then \cite{23} 
$$ \int_{\Omega} f(|u(x)|)dx =\int_{0}^{\infty}f(u^{*}(t))dt.$$
\section{A priori estimates}
Here we provide our a priori estimates for the approximate solutions to problem \eqref{01}.\\
\textbf{Approximating problems.} Let $n\in I\!\!N,$
\begin{equation}\label{31}
\left\{
\begin{array}{lll}
&-\mbox{div}\left( a(x,T_{n}(u_{n}),\nabla u_{n})\right)=f_{n} h_{n}(u_{n})  &\mbox{ in }  \Omega \\
&u_{n}=0& \mbox{ on }  \delta\Omega
\end{array}
\right.
\end{equation}
where  $f_{n}=T_{n}(f)$. Moreover, defining  $h(0):= \lim_{s\rightarrow0}h(s)$, we set 
\begin{equation}\label{A2}
h_{n}(s)=\left\{
\begin{array}{ll}
T_{n}(h(s)) \,\,\,\, \mbox{for} \,\,\,\,\,\, s>0,\\
\min(n,h(0)) \,\,\,\,\,\,\,\,  \mbox{otherwise},
\end{array}
\right.
\end{equation}
where $T_{n}(h(u_{n}))\leq \frac{C}{(|u_{n}|+\frac{1}{n})^{\gamma}}$ and  $T_{n}(s)=\max\{-n, \min \{n,s\} \}.$
The right hande side of \eqref{31} is non negative, that $u_{n}$ is nonnegative. Observe that we have "truncated" the degenerate coercivity of the operator term and the singularity of the right hand  side. The weak formulation of \eqref{31} is 
\begin{equation}\label{33}
\int_{\Omega} a(x,T_{n}(u_{n}),\nabla u_{n}) \nabla \varphi dx =\int_{\Omega} f_{n} h_{n}(u_{n})\varphi dx \forall \varphi\in L^{\infty}(\Omega)\cap W_{0}^{1,p}(\Omega).
\end{equation}    
\begin{proposition} \label{pro1}
	For each $n\in \textbf{N}$ there exists $u_{n}\in W_{0}^{1,p}(\Omega)\cap L^{\infty}(\Omega)$ weak solution of problem \eqref{31}.
\end{proposition}    
\begin{proof}
	The proof is based on standard Schauder's fixed point argument. Let $n\in I\!\!N $ be fixed  and $v\in L^{p}(\Omega)$ be fixed. we know that the following non singular problem 
	\begin{equation*}
	\left\{
	\begin{array}{ll}
	-\mbox{div}\left( a(x,T_{n}(w),\nabla w)\right)  =f_{n}h_{n}(v)  \, in \,\,\,\,\Omega \\
	w=0 \,\,\,\, on\,\,  \delta\Omega
	\end{array}
	\right.
	\end{equation*}
	has a unique solution $w\in W_{0}^{1,p}(\Omega)\cap L^{\infty}(\Omega)$ follows from the classical results (see \cite{24} and \cite{02}). In particular, it is well defined a map 
	$$ G:L^{p}(\Omega) \rightarrow L^{p}(\Omega),$$
	where $G(v)=w$. Again,thanks to regularity of the datum $h_{n}(v) f_{n}$, we cane take $w$ as test function and obtain 
	\begin{equation}\label{34}
	\int_{\Omega} a(x,T_{n}(w),\nabla w) \nabla w = \int_{\Omega}  f_{n}h_{n}(v)w,
	\end{equation}
	then, it follows from \eqref{12}
	$$
	\alpha\int_{\Omega} \frac{|\nabla w |^{p}}{(1+n)^{\theta (p-1)}}dx \leq\int_{\Omega} \frac{|\nabla w |^{p}}{(1+|T_{n}(w)|)^{\theta (p-1)}}dx \leq n^{2}\int_{\Omega} | w | dx
	$$
	using the poincar\'e inequality we have
	$$
	\int_{\Omega} \frac{|\nabla w |^{p}}{(1+n)^{\theta (p-1)}}dx \leq c_{1} n^{2}\int_{\Omega} |\nabla w | dx,
	$$
	by H\"{o}lder's  inequality on the right hand side, we obtain
	\begin{equation*}
	\int_{\Omega} |\nabla w |^{p} dx \leq c_{1}(1+n)^{\theta(p-1)}n^{2}\int_{\Omega}  |\nabla w | dx 
	\leq c(n) |\Omega |^{\frac{1}{p^{'}}}  \left( \int_{\Omega} |\nabla w |^{p} dx  \right)^{\frac{1}{p}} 
	\end{equation*}
	we deduce $$ \displaystyle\int_{\Omega} |\nabla w |^{p} dx  \leq c(n)^{p^{\prime}}  |\Omega| ,$$
	Using the Poincar\'{e} inequality on the left hand side 
	$$  ||w||_{L^{p}(\Omega)}\leq c(n,|\Omega |) (=c^{\frac{p^{\prime}}{p}}(n)  |\Omega |^{\frac{1}{p}}),$$ 
	where $c(n,|\Omega |)$ is a positive constant independent form $v$ and $w$, thus, we have that the ball $B$ of $L^{p}(\Omega)$ of radius $c(n,|\Omega |)$  is invariant for the map $G.$
	\par Now we prove that the map $G$ is continuous in $B.$ 
	Let us choose a sequence $v_{k}$ that converges strongly to $v$ in  $L^{p}(\Omega)$, the by dominated convergence theorem 
	$$ f_{n}h_{n}(v_{k}) \rightarrow f_{n}h_{n}(v) \mbox{ in } L^{p}(\Omega), $$
	then we need to prove that $G(v_{k})$ converge to $G(v)$ in $L^{p}(\Omega).$ By compactness we already know  that the sequence  $w_{k}=G(v_{k})$ converge to some function  $w$ in $L^{p}(\Omega).$ We only need to prove that $w=G(v).$
	Firstly, we have the datum $f_{n}h_{n}(v_{k})$ are bounded, we have that $w_{k}\in L^{\infty}(\Omega) $ and there exists a positive constant $d$, independent of $v_{k}$ and $w_{k}$  (but possibly depending on $n$ ), such that $||w_{k}||_{L^{\infty}(\Omega)}\leq d.$ We know the sequence $w_{k}$ is bounded in $W_{0}^{1,p}(\Omega)$. Hence, by uniqueness, one deduces that $G(v_{k})$ converge to $G(v)$ in $L^{p}(\Omega).$
	Lastly we need to check that the set  $G(B)$ is relatively  compact, 
	Let $v_{k}$ be a bounded sequence in $B.$ and let $ w_{k} =G(v_{k}).$ we proved before that $$\int _{\Omega} |\nabla w|^{p}dx=\int _{\Omega} |\nabla G(v)|^{p}dx \leq c(n,|\Omega|),$$ 
	for any $v\in L^{p}(\Omega)$, then for $v=v_{k}$ we obtain
	$$\int _{\Omega} |\nabla w_{k}|^{p}dx=\int _{\Omega} |\nabla G(v_{k})|^{p}dx \leq c(n,|\Omega|),$$ 
	so that  $G(v)$ is relatively compact in $L^{p}(\Omega)$ by Rellich-kondrachov Theorem. We can then apply Schauder fixed point theorem  there exist a fixed point of the map $G$, say $u_{n}$  will exist in $B$ such that  $ G(u_{n})=u_{n}$ and we will have that $u_{n}\in W_{0}^{1,p}(\Omega)\cap L^{\infty}(\Omega)$ is solution of problem \eqref{31}. 
\end{proof}
\begin{theorem}\label{Athm5}
	Let $f$ be in $L^{m}(\Omega)$ with $m >N\slash p $, $0\leq\theta\leq 1 $ and let $u_{n}$ be solution of \eqref{31}. Then the norm of $u_{n}$ in $L^{\infty}(\Omega)$. Indeed, we have  
	\begin{equation}\label{310}
	\Arrowvert u_{n}\Arrowvert_{L^{\infty}(\Omega)}<\mathcal{B}^{-1}\left[ \frac{C^{\frac{1}{p-1}} |\Omega|^{\frac{P^{'}}{N}-\frac{p^{'}}{pm}}}{(NC_{N}^{\frac{1}{N}})^{p^{'}}} \frac{Nm(p-1)}{pm-N}\Arrowvert f \Arrowvert_{L^{m}(\Omega)}^{\frac{p^{'}}{p}}    \right],  
	\end{equation}
	where $\mathcal{B}^{-1}$ denotes the inverse function of $\mathcal{B}.$ Furthermore, the norm of $u_{n}$ in $W_{0}^{1,p}(\Omega)$ is bounded by a constant continuously depending on the norm of $f$ in $(L^{m}(\Omega))^{N}.$
\end{theorem}
\begin{proof}
	For $\varepsilon> 0$ and $t>1$, we use in the formulation \eqref{33}. Let the test function $v=T_{\varepsilon}(G_{t}(u_{n}))$ where $\{t<\arrowvert u_{n} \arrowvert <t+\varepsilon\}$ denotes the test set\\  $\{x\in\Omega :t<\arrowvert u_{n}(x) \arrowvert <t+\varepsilon\}$
	Assumption \eqref{12} yields
	\begin{align*}
	&\alpha\int_{{\{t<\arrowvert u_{n} \arrowvert <t+\varepsilon\}}} \frac{|\nabla u_{n}|^{p}}{(1+| u_{n}|)^{\theta(p-1)}}dx 
	\leq \epsilon \int_{{\{t<|u_{n}(x)| \}}} f_{n} h_{n}(u_{n})dx\\
	&\leq \varepsilon  \sup_{u_{n}\in [t,+\infty]}(h_{n}(u_{n}))\int_{{t<| u_{n}(x)|}}  f_{n} dx\\
	& \leq  \epsilon \sup_{u_{n}\in [t,+\infty]}\left(\frac{C}{(|u_{n}(x)|+\frac{1}{n})^{\gamma}}\right)\int_{{\{t<|u_{n}(x)| \}}} fdx, 
	\end{align*}
	in the set $\{t< |u_{n}(x)| \},$ we have that $|u_{n}(x)|+\frac{1}{n}> t>1$ and dividing both sides by $\varepsilon$ we get 
	$$\frac{\alpha}{\varepsilon}\int_{{\{t<\arrowvert u_{n}(x) \arrowvert <t+\varepsilon\}}} \frac{|\nabla u_{n}|^{p}}{(1+| u_{n}|)^{\theta(p-1)}}dx \leq 
	C \int_{t< |u_{n}(x)|}  f dx.$$
	The above inequality and H\"{o}lder's inequality 
	\begin{align}
	&\left( \frac{\alpha}{\varepsilon}\int_{{\{t<\arrowvert u_{n}(x) \arrowvert <t+\varepsilon\}}} \frac{|\nabla u_{n}|}{(1+| u_{n}|)^{\theta(p-1)}}dx \right)^{p}\nonumber \\
	&\leq C \left( \frac{\alpha}{\varepsilon}\int_{{\{t<\arrowvert u_{n}(x) \arrowvert <t+\varepsilon\}}} \frac{ 1}{(1+| u_{n}|)^{\theta(p-1)}}dx \right)^{p-1}\int_{\{t<|u_{n}(x)| \}}  f dx.\label{311}
	\end{align}
	We can pass to the limit as $\varepsilon$ goes to $0^{+}$ in \eqref{311} to get, after simplification
	\begin{equation}\label{312}
	\frac{\alpha}{(1+t)^{\theta(p-1)}}\left(  \frac{d}{dt} \int_{|u_{n}|\leq t} | \nabla u_{n}|dx\right) ^{p}\leq  C (- \mu^{\prime}_{u_{n}(t)})^{p-1}\left(  \int_{0}^{\tau} | f^{\ast}_{n}(\tau)|d\tau\right).
	\end{equation}
	On the other hand, from Fleming-Rishel Coera Formula and isoperimetric inequality we have for almost every $t>0$
	\begin{equation}\label{313}
	N C_{N}^{\frac{1}{N}}(\mu_{u_{n}}(t))^{\frac{N-1}{N}} \leq  \frac{d}{dt}\int_{|u_{n}|\leq t} |\nabla u_{n}|dx,
	\end{equation}
	where $C_{N}$ is the measure of the unit ball in $I\!\!R^{N}.$ Using the H\"{o}lder's inequality we obtain that for almost every $t>0$, then \eqref{312} and \eqref{313} give
	\begin{equation}\label{314}
	\frac{\alpha^{\frac{1}{p-1}}}{(1+t)^{\theta}}\leq \frac{(-\mu^{\prime}_{u_{n}}(t))C^{\frac{1}{p-1}}}{\left(N C_{N}^{\frac{1}{N}}(\mu_{u_{n}}(t))^{1-\frac{1}{N}}\right)^{p^{\prime}}}\left(    \int_{0}^{\mu_{u_{n}}(t)}  f^{\ast}(\tau)d\tau\right)^{\frac{p^{\prime}}{p}}.
	\end{equation}
	Using the properties of rearrangements one easily obtains 
	\begin{equation}\label{315}
	\frac{-d}{d\sigma} \mathcal{B}(u^{\ast}_{n}(\sigma)) \leq 
	\frac{C^{\frac{1}{p-1}}}{\left( N C_{N}^{\frac{C}{N}}(\sigma)^{1-\frac{1}{N}}\right)^{p^{'}}}
	\left(\int_{0}^{\sigma}  f^{\ast}(\tau)d\tau \right)^{\frac{1}{p-1}}.
	\end{equation} 
	If integrate \eqref{315} between $\sigma$ and $|\Omega|$, we have
	$$ \mathcal{B}(u^{\ast}_{n}(\sigma)) \leq 
	\frac{C^{\frac{1}{p-1}}}{\left( N C_{N}^{\frac{1}{N}}\right) ^{p^{\prime}}}
	\int_{\sigma}^{|\Omega|} \left(\int_{0}^{\rho}  f^{\ast}(\tau)d\tau \right) ^{\frac{p^{\prime}}{p}}\frac{d\rho}{\rho^{p^{'}(1-\frac{1}{N})}}.$$
	Immediately we get\eqref{310} by evaluating $\mathcal{B}(u^{\ast}_{n}(0))$.
	Let us denote in what by $  c_{\infty}$  the constant on the right in \eqref{310}, that is 
	\begin{equation}\label{316}
	||u_{n}||_{\infty}\leq c_{\infty},
	\end{equation}
	it is easy to get an estimation in $W_{0}^{1,p}(\Omega)$. Taking $u_{n}$ as test function in formulation \eqref{19} then using \eqref{12},\eqref{315} and H\"{o}der inequality, we get 
	\begin{align}
	&b^{p}(c_{\infty}) \int_{\Omega} |\nabla u_{n}|^{p}dx \leq \int_{\Omega} 
	f u_{n}^{1-\gamma} dx \leq ||u_{n}^{1-\gamma}||_{L^{\infty} (\Omega)} \int_{\Omega} f dx \leq c_{\infty} \int_{\Omega}f dx\nonumber\\
	& \leq c_{\infty} |\Omega|^{1-\frac{1}{m}} ||f||_{L^{m}(\Omega)},\nonumber
	\end{align}
	then 
	\begin{equation}\label{317}
	\int_{\Omega}|\nabla u_{n}|^{p} dx \leq \frac{c_{\infty} |\Omega|^{1-\frac{1}{m}}}{b^{p}(c_{\infty})} ||f||_{L^{m}(\Omega)}.
	\end{equation}
\end{proof}
\begin{theorem}\label{Athm6}
	On suppose that $0< \gamma < \theta(p-1)+1 $ and 
	$$
	\frac{Np}{Np-(N-p)[\theta(p-1)+1-\gamma]}\leq m <N\slash p,
	$$
	let $$r=\frac{Nm[(p-1)(1-\theta)+\gamma]}{N-pm} .$$
	Then, the solution  $u_{n}$ to \eqref{33} are uniformly  bounded in $L^{r}(\Omega)\cap W_{0}^{1,p}(\Omega).$
\end{theorem} 
\begin{proof}
	Let us choose $(1+u_{n})^{\nu}-1$ as a test function by the hypotheses on $a$, one has 
	\begin{align}
	& \nu \left( \frac{p}{(p-1)(1-\theta)+ \nu}\right)^{p}\int_{\Omega}\big\vert\nabla [(1+u_{n})^{\frac{-\theta(p-1)+ \nu+p-1}{p}}-1] \big\vert^{p}\,dx\nonumber \\
	&= \nu \int _{\Omega}\frac{|\nabla u_{n}|^{p}}{(1+u_{n})^{\theta(p-1)- \nu+1}} dx
	\leq \nu\int_{\Omega} \frac{|\nabla u_{n}|^{p}}{(1+T_{n}(u_{n}))^{\theta(p-1)}}(1+u_{n})^{ \nu-1}dx\nonumber\\
	& \leq \int_{\Omega} \frac{T_{n}(f)}{(u_{n}+\frac{1}{n})^{\gamma}} \big((u_{n}+1)^{ \nu}-1\big)\,dx
	\leq C+C\int_{\Omega} \frac{|f|}{(u_{n}+1)^{- \nu+\gamma}} dx.\label{y2}
	\end{align}
	By Sobolev's inequality on the left hand side and H\~{o}lder's
	inequality on the right one we have 
	\begin{align} &\left( \int_{\Omega} \big((1+u_{n})^{\frac{-\theta(p-1)+ \nu+p-1}{p}}-1\big)^{p^{\ast}}dx \right)^{\frac{p}{p^{\ast}}}\nonumber \\
	&\leq C ||f||_{L^{m}(\Omega)}\left(\int_{\Omega}(u_{n}+1)^{m^{\prime}( \nu-\gamma)}dx \right)^{\frac{1}{m^{\prime}}}.\label{y3}
	\end{align}
	Let $\nu$ be such that 
	$$\frac{-\theta(p-1)+ \nu+p-1}{N-p}N=(\frac{ \nu-\gamma}{m-1})
	m$$ and $ \frac{p}{p^{\ast}}> \frac{1}{m^{\prime}}$, that is 
	$$  \nu=\frac{N(m-1)(1-\theta)(p-1)+\gamma m(N-p)}{N-pm}$$
	and $m< \frac{N}{p} $, we observe that  
	$$ \frac{p^{\ast}}{p}(-\theta(p-1)+ \nu+p-1)=\frac{mN}{N-pm}[(p-1)(1-\theta)+\gamma]=r>1$$  This implies that $u_{n}$ is bounded in $ L^{r}(\Omega).$\\
	By \eqref{y2}, \eqref{y3} and $\mu \geq 1+\theta(p-1) \left(\Leftrightarrow \frac{Np}{Np-(N-p)[\theta(p-1)+1-\gamma]}\leq m \right),$ we get
	$$ \int_{\Omega} |\nabla u_{n}|^{p} dx \leq  \int _{\Omega}\frac{|\nabla u_{n}|^{p}}{(1+u_{n})^{\theta(p-1)- \nu+1}} dx  \leq C  ||f||_{L^{m}}  \int_{\Omega} |u_{n}|^{r} dx\leq Cst.$$
\end{proof}
\begin{theorem}\label{Athm7}
	On suppose that $0<\gamma < \theta(p-1)+1 $ and  \eqref{115} holds true. Let $\sigma$ be as in \eqref{116} then the solution $u_{n}$ to \eqref{33} are uniformly bounded in $W_{0}^{1,\sigma}(\Omega)\cap L^{r}(\Omega).$
\end{theorem}
\begin{proof}
	Let us choose $(1+u_{n})^{\lambda}-1$ with 
	$\lambda=\frac{N(m-1)(1-\theta)(p-1)+\gamma m(N-p)}{N-pm}$
	as a test function in \eqref{33} with the summer arguments as before we have  
	\begin{align*}
	&\left(\int_{\Omega} [(1+u_{n})^{\frac{-\theta(p-1)+\lambda+p-1}{p}}-1]^{p^{\ast}}dx\right) ^{\frac{p}{p^{\ast}}} \leq 
	C\int_{\Omega} \frac{|\nabla u_{n}|^{p}}{(1+u_{n})^{\theta(p-1)-\lambda+1}} dx \\
	&\leq C||f||_{L^{m}(\Omega)}\left(\int_{\Omega}(1+u_{n})^{m^{\prime}(\lambda-\gamma)}dx\right)^{\frac{1}{m^{\prime}}}.
	\end{align*}
	As above, we infer that $u_{n}$ is bounded in  $ L^{\frac{N((p-1)(1-\theta)+\lambda)}{N-p}}(\Omega).$
	We observe that $\theta(p-1)-\lambda+1 >0$
	and  $1<\sigma=\frac{Nm[(p-1)(1-\theta)+\gamma]}{N-m(\theta(p-1)+1-\gamma)},$ by the assumptions on  $m$, writing 
	$$\int_{\Omega}| \nabla u_{n}|^{\sigma} dx =\int_{\Omega}\frac{| \nabla u_{n}|^{\sigma}}{(1+u_{n})^{\frac{\theta (p-1)-\lambda+1}{p}}} (1+u_{n})^{\frac{\theta (p-1)-\lambda+1}{p}} dx $$
	and using H\"{o}lder's inequality with exponent $\frac{p}{\sigma}$, we obtain 
	$$\int_{\Omega}| \nabla u_{n}|^{\sigma} dx\leq 
	\left[ \int_{\Omega}\frac{| \nabla u_{n}|^{\sigma}}{(1+u_{n})^{\theta (p-1)-\lambda+1}}dx\right]^{\frac{\sigma}{p}}
	\left[\int_{\Omega}(1+u_{n})^{\sigma \frac{\theta (p-1)-\lambda+1}{p-\sigma}}
	dx\right]^{\frac{p-\sigma}{p}}.$$
	The above estimates imply that the sequences $u_{n}$ is bounded
	in $ W_{0}^{1,\sigma}(\Omega)$ if 
	$$ \sigma \frac{\theta (p-1)-\lambda+1}{p-1}= \frac{N[(p-1)(1-\theta)+\lambda]}{N-p},$$
	that is 
	$$\sigma=\frac{Nm[(p-1)(1-\theta)+\gamma]}{N-m[\theta(p-1)+1-\gamma]}.$$
	By virtue of $\lambda <1+\theta(p-1) \mbox{ or } \sigma <p,$ therefore we have $m<Np\slash[ Np-(N-p)(\theta(p-1)+1-\gamma)].$
\end{proof}
\begin{lemma}\label{lm1}
	Assume that $\gamma =\theta(p-1)+1$ and $f\in L^{1}(\Omega)$
	the solution $u_{n}$ to \eqref{31} are uniformly bounded in 
	$W_{0}^{1,p}(\Omega)$.
\end{lemma} 
\begin{proof}
	Let us choose $ (1+u_{n})^{\theta (p-1)+1}-1$ as a test function in \eqref{33} we have 
	$$ \theta(p-1) \int_{\Omega}\frac{| \nabla u_{n}|^{p}}{(1+T_{n}(u_{n}))^{\theta (p-1)}} (1+u_{n})^{\theta (p-1)} dx \leq C \int_{\Omega} f dx.$$
	The previous estimate implies the sequence $u_{n}$ is bounded 
	in $W_{0}^{1,p}(\Omega)$.
\end{proof}
\begin{lemma}\label{lm2}
	Assume that $\gamma > \theta(p-1)+1$ and $f\in L^{1}(\Omega)$ 
	then the solution $u_{n}$ to \eqref{31} are such that 
	$u_{n}^{\frac{\gamma+(p-1)(1-\theta)}{p}}$ is uniformly bounded in  	$W_{0}^{1,p}(\Omega)$, $u_{n}$ uniformly bounded in $
	L^{\frac{\gamma+(p-1)(1-\theta)}{p}p^{\ast}}(\Omega).$ 
\end{lemma} 
\begin{proof}
	If we choose $u_{n}^{\gamma}$ as test function and use the hypotheses on $a$ we get 
	\begin{align*}
	&\alpha\gamma\left(\frac{p}{\gamma+(p-1)(1-\theta)}\right)^{p} \int_{\Omega} |\nabla u_{n}^{\frac{\gamma+(p-1)(1-\theta)}{p}}|^{p}dx \\
	&=\alpha\gamma \int_{\Omega} |\nabla u_{n}|^{p} u_{n}^{\gamma-1-\theta (p-1)} dx\leq \int_{\Omega} f dx.
	\end{align*}
	This prove that the sequence  $ u_{n}^{\frac{\gamma+(p-1)(1-\theta)}{p}}$  is bounded in $W_{0}^{1,p}(\Omega).$ Sobolev's inequality an the left  hand side applied to 
	$ u_{n}^{\frac{\gamma+(p-1)(1-\theta)}{p}}$ gives 
	$$ \int_{\Omega}  u_{n}^{\frac{\gamma+(p-1)(1-\theta)}{p}p^{\ast}} dx\leq C.$$
\end{proof}
\section{Proof of the results} 
In this section we are going to combine the results of section 2 and 3  in order to prove Theorem \ref{Athm1}, Theorem \ref{Athm2} and Theorem \ref{Athm3}. 
\begin{proof}[Proof of Theorem \ref{Athm1}]\qquad \\
	\textbf{Step 1}:
	We prove that 
	\begin{equation}\label{l1}
	\displaystyle\lim_{n\rightarrow+\infty} \int_{\Omega} h_{n}(u_{n})f_{n}\varphi dx =\int h(u)f\varphi dx,
	\end{equation}
	for all non negative $\varphi \in W_{0}^{1,p}(\Omega)\cap L^{\infty} (\Omega)$.
	First we observe that from the Young inequality and the hypotheses in \eqref{14}, one  gets 
	\begin{align*}
	&\displaystyle\int_{\Omega} h_{n}(u_{n}) f_{n}  \varphi =
	\int_{\Omega}a(x,T_{n}(u_{n}),\nabla u_{n})\nabla \varphi dx
	\leq \int_{\Omega} a_{0}(x) \nabla \varphi dx\\ 
	&+\int_{\Omega}|u_{n}|^{p-1} \nabla \varphi dx 
	+\int_{\Omega}|\nabla u_{n}|^{p-1} \nabla \varphi dx 
	\leq \int_{\Omega} a_{0}(x) \nabla \varphi dx
	+\frac{p-1}{p}\int_{\Omega}|u_{n}|^{p}\\
	&+\frac{1}{p}\int_{\Omega}|\nabla \varphi |^{p}dx +\frac{p-1}{p}\int_{\Omega}|\nabla u_{n}|^{p} dx+ \frac{1}{p}\int_{\Omega}|\nabla \varphi |^{p} dx 
	\leq \frac{1}{p^{'}}\int_{\Omega} a_{0}(x)^{p^{'}}dx+\\
	&\frac{1}{p} \int_{\Omega}|\nabla \varphi |^{p}dx
	+\frac{p-1}{p}\int_{\Omega}|u_{n}|^{p}
	+\frac{1}{p}\int_{\Omega}|\nabla \varphi |^{p}dx +\frac{p-1}{p}\int_{\Omega}|\nabla u_{n}|^{p} dx\\
	&+ \frac{1}{p}\int_{\Omega}|\nabla \varphi |^{p} dx 
	\leq c+c\left[ \int_{\Omega}|\nabla \varphi |^{p}dx+
	\int_{\Omega}| u_{n}|^{p} dx+ \int_{\Omega}|\nabla u_{n}|^{p}dx \right],
	\end{align*}
	then 
	\begin{equation}\label{41}
	\int_{\Omega} h_{n}(u_{n}) f_{n}  \varphi \leq c+c [||\varphi||_{W_{0}^{1,p}(\Omega)}+||u_{n}||_{W_{0}^{1,p}(\Omega)}].
	\end{equation}
	From now we consider a non negative $ \varphi \in W_{0}^{1,p}(\Omega) \cap L^{\infty}(\Omega)$.
	An application of the Fatou Lemma in \eqref{41}	
	with respect to $ n \,\,$ gives 
	\begin{equation}\label{42}
	\int_{\Omega} h(u)f \varphi \leq c,
	\end{equation} 
	where $c$ does not depend on $n$. Hence $fh(u)\varphi \in L^{1}(\Omega)$ for any non negative $\varphi\in W^{1,p}_{0}(\Omega)\cap L^{\infty}(\Omega)$. As a consequence, if $h(s)$ is unbounded as $s$
	tends to $ 0$, we deduce that
	\begin{equation}\label{u0}
	\{u=0\}\subset \{f=0\},
	\end{equation}
	up to a set of zero Lebesgue measure.\\
	From now on, we assume that $h(s)$ is unbounded as $s$ tends to $0$. Let $\varphi$ be a non negative function in $ W^{1,p}_{0}(\Omega)\cap L^{\infty}(\Omega)$, choosing it as test function in the weak formulation of \eqref{31}, we have 
	\begin{equation}\label{310}
	\int_{\Omega} a(x,T_{n}(u_{n}),\nabla u_{n}) \nabla \varphi dx =\int_{\Omega} f_{n} h_{n}(u_{n})\varphi dx .
	\end{equation}
	We want to pass to the limit in the right hand side of \eqref{310} as $n$ tends to infinity. we fix  $\delta >0 $, and we decompose the right hand side in the following way 
	\begin{equation}\label{311}
	\int_{\Omega} h_{n}(u_{n}) f_{n}  \varphi dx =\int _{{u_{n}\leq \delta}} h_{n}(u_{n}) f_{n}  \varphi dx+
	\int_{{u_{n}> \delta}} h_{n}(u_{n}) f_{n}  \varphi dx.
	\end{equation}
	Therefore we have, thanks to Lemma1.1 contained in \cite{25}, that $V_{\delta}(u_{n})$  belongs to $W_{0}^{1,p}(\Omega)$,
	where $ V_{\delta}$ is defined by 
	\begin{equation}\label{N1}
	V_{\delta}(s)=\left\{
	\begin{array}{ll}
	1 \,\,\,\,\,\,\,\,\,\,\,\,\,\,\,\,\,\,\,\,\,\,\,\,  s\leq \delta \\
	\frac{2\delta-s}{\delta} \,\,\,\,\,\,\,\,\, \delta<s<2\delta,\\
	0 \,\,\,\,\,\,\,\,\,\,\,\,\,\,\,\,\,\,\,\,\,\,\,\,\,  s\geq 2\delta.
	\end{array}
	\right.
	\end{equation}
	So we take it is test function in the weak formulation of \eqref{31}, using \eqref{N1}, \eqref{12} and \eqref{14} we obtain
	\begin{align*}
	&\int _{\{u_{n} \leq \delta\}} h_{n}(u_{n}) f_{n}  \varphi dx\leq \int _{\Omega} h_{n}(u_{n}) f_{n} V_{\delta}(u_{n}) \varphi dx\\
	&=\int_{\Omega} a(x,T_{n}(u_{n}),\nabla u_{n})\nabla \varphi V_{\delta}(u_{n})dx-\frac{1}{\delta}\int_{\{\delta<u_{n}<2\delta\}} a(x,T_{n}(u_{n}),\nabla u_{n}) \varphi \nabla u_{n}dx,
	\end{align*}
	by using \eqref{12} and \eqref{14}, we have 
	\begin{align*}
	&\int _{\{u_{n} \leq \delta\}} h_{n}(u_{n}) f_{n}  \varphi dx\leq \beta \int_{\Omega} \big[ a_{0}(x)+|u_{n}|^{p-1}+|\nabla u_{n}|^{p-1}\big]\nabla \varphi V_{\delta}(u_{n})dx \\
	&-\frac{1}{\delta (n+1)^{\theta(p-1)}}\int_{\{\delta<u_{n}<2\delta\}}  |\nabla u_{n}|^{p}\varphi dx\\
	&\leq \beta \int_{\Omega} \big[ a_{0}(x)+|u_{n}|^{p-1}+|\nabla u_{n}|^{p-1}\big]\nabla \varphi V_{\delta}(u_{n})dx.
	\end{align*}
	Using that $V_{\delta}$ is bounded we deduce that $|\nabla u_{n}|^{p-1} \nabla\varphi V_{\delta}(u_{n})$ converges to 
	$|\nabla u|^{p-1}\nabla \varphi V_{\delta}(u)$ weakly in $L^{p'}(\Omega)^{N}$ as $n$ tends to infinity. This implies that
	\begin{equation}\label{121}
	\lim_{n\rightarrow +\infty}\int _{\{u_{n} \leq \delta\}} h_{n}(u_{n}) f_{n}  \varphi dx\leq \beta \int_{\Omega} \big[ a_{0}(x)+|u|^{p-1}+|\nabla u|^{p-1}\big]\nabla \varphi V_{\delta}(u)dx.
	\end{equation}
	Since $V_{\delta}(u)$ converges to $\chi_{\{u=0\}}$ a.e in $\Omega$ as $\delta$ tends to $0$ and since $u \in W_{0}^{1,p}(\Omega)$, then $ \big[ a_{0}(x)+|u|^{p-1}+|\nabla u|^{p-1}\big]\nabla \varphi V_{\delta}(u)$ converges to $0$ a.e. in $\Omega$ as $\delta$ tends to $0$. Applying the Lebesgue Theorem on the right hand side of \eqref{121} we obtain that 
	\begin{equation}\label{131}
	\lim_{\delta\rightarrow 0^{+}} \lim_{n\rightarrow +\infty}\int _{\{u_{n} \leq \delta\}} h_{n}(u_{n}) f_{n}  \varphi dx=0.
	\end{equation}	
	As regards the second term in the right hand side of  \eqref{311} we have 
	\begin{equation}\label{314}
	0\leq	h_{n}(u_{n}) f_{n}  \varphi \chi_{\{u_{n}>\delta\}} \leq \sup_{s\in ]\delta,\infty)} [h(s)]f \varphi \in L^{1}(\Omega),
	\end{equation}
	we remark that we need to choose $\delta \neq \{\eta ;  |u=\eta|> 0 \},$ which is at most a countable set. As a consequence $\chi_{\{u_{n}>\delta\}}$  converges to  $\chi_{\{u>\delta\}} $ a.e in $\Omega$, we deduce first that 
	$ h_{n}(u_{n}) f_{n} \chi_{\{u_{n}>\delta\}}  \varphi\,\,\,\,  \mbox{converges to} \,\,\,h(u) f \chi_{\{u>\delta\}} \varphi $ strongly in $L^{1}(\Omega)$ as $n$ tends to infinity, then, since $ h(u) f \chi_{\{u>\delta\}} \varphi$ belongs 
	to $L^{1}(\Omega)$, that $fh(u) \chi_{\{u>\delta\}} \varphi$ converges to $fh(u) \chi_{\{u>0\}} \varphi$ strongly in $L^{1}(\Omega)$ as $ \delta$ tend to $0$. \\and then, once again by the Lebesgue Theorem, one gets
	\begin{equation}\label{141}
	\lim_{\delta \rightarrow 0^{+}}\lim_{n \rightarrow +\infty}\int _{\{u_{n} > \delta\}} h_{n}(u_{n}) f_{n}  \varphi dx
	=\int _{\{u > 0\}} h(u) f \varphi dx.
	\end{equation}
	By \eqref{141} and \eqref{131}, we deduce that 
	\begin{equation}\label{44}
	\lim_{n \rightarrow +\infty}\int_{\Omega} h_{n}(u_{n}) f_{n}  \varphi dx
	=\int_{\Omega} h(u) f \varphi dx\,\,\,\,\forall 0\leq \varphi \in W_{0}^{1,p}(\Omega)\cap L^{\infty}(\Omega).
	\end{equation} 
	Moreover, decomposing any $\varphi=\varphi^{+}-\varphi^{-}$, and using that \eqref{44} is linear in $\varphi$, we deduce that \eqref{44} holds for every $\varphi \in W_{0}^{1,p}(\Omega)\cap L^{\infty}(\Omega)$.
	We treated $h(s)$ unbounded as $s$ tends to $0$, as regards bounded function $h$ the proof is easier and only difference
	deals with the passage to the limit in the left hand side of \eqref{44}. We can avoid introducing $\delta$ and we can substitute \eqref{314} with 
	$$0\leq f_{n}h_{n}(u_{n})\varphi \leq f||h||_{L^{\infty}(\Omega)}\varphi.$$
	Using the same argument above we have that $ f_{n}h_{n}(u)\varphi$ converges to $ fh(u)\varphi$ strongly in $ L^{1}(\Omega)$
	as $n$ tends to infinity.
	This concludes \eqref{l1}.\\
	\textbf{Step 2}: Thanks to \eqref{317}, the sequence ${u_{n}}$ is bounded in $W_{0}^{1,p}(\Omega)$. Therefore, there exist a subsequence of ${u_{n}}$ still denoted by ${u_{n}}$, and a measurable function $u$ such that 
	\begin{equation}\label{45}
	u_{n}\rightharpoonup u \,\,\,\,\,weakly\,\, in\,\, W_{0}^{1,p}(\Omega) \,\,and \,\,\,a.e \,\,\,in \,\,\,\Omega.
	\end{equation}
	We shall prove that 
	\begin{equation}\label{qq1}
	u_{n} \longrightarrow u \,\,\,\,\, \mbox{strongly in } \,\,\,\, W_{0}^{1,p}(\Omega).
	\end{equation}
	We take $ u_{n}-u$ test function in the weak formulation of \eqref{33}, we obtain for $n>c_{\infty}$
	\begin{equation}\label{36}
	\int_{\Omega}a(x,u_{n},\nabla u_{n}) \nabla (u_{n}-u)dx = \int_{\Omega} f_{n} h_{n}(u_{n}) (u_{n}-u) dx,
	\end{equation}
	the right hand side tends to zero when $n $ tends to infinity. On the other hand we write 
	\begin{align}\label{37}
	&\int_{\Omega}a(x,u_{n},\nabla u_{n}) -a(x,u_{n},\nabla u)\nabla (u_{n}-u)\nonumber \\
	&=\int_{\Omega}a(x,u_{n},\nabla u_{n}) \nabla (u_{n}-u)dx-\int_{\Omega} a(x,u_{n},\nabla u)\nabla (u_{n}-u),
	\end{align}
	by \eqref{34} one has 
	$$ \lim_{n\rightarrow +\infty}\int_{\Omega}a(x,u_{n},\nabla u_{n}) \nabla (u_{n}-u)dx= 0,$$
	As regards the second term on the right in \eqref{37} and see step 1 in the proof of Theorem \ref{Athm1}, using \eqref{14} and Vitali's Theorem  we obtain that 
	$$a(x,u_{n},\nabla u)\longrightarrow  a(x,u,\nabla u)\,\,\,\, \mbox{strongly in}\,\,\, (L^{p'}(\Omega))^{N}.$$
	Therefore, we obtain 
	\begin{equation}\label{38}
	\lim_{n\rightarrow +\infty}\int_{\Omega}(a(x,u_{n},\nabla u_{n})- a(x,u_{n},\nabla u) )\nabla (u_{n}-u)dx=0,
	\end{equation}
	thanks to \eqref{15}, the integrand function in the left hand side in \eqref{38} is non negative, therefore
	$$(a(x,u_{n},\nabla u_{n})- a(x,u_{n},\nabla u)) \nabla (u_{n}-u)\longrightarrow 0 \,\,\,\mbox{strongly} \,\, \mbox{in}\,\, L^{1}(\Omega). $$
	Thus, up a subsequence still indexed by $u_{n}$, one has 
	$$a((x,u_{n},\nabla u_{n})- a(x,u_{n},\nabla u)) \nabla (u_{n}-u)\longrightarrow 0, $$
	for  almost  every $x$ in $\Omega$, there exists a subset $Z$ of $\Omega$ zero measure, such that for all $x$ in $\Omega \backslash Z$ we have 
	\begin{equation}\label{39}
	D_{n}(x)=(a(x,u_{n}(x),\nabla u_{n}(x))- a(x,u_{n}(x)_,\nabla u(x) \nabla (u_{n}-u)(x)))\longrightarrow 0 
	\end{equation}
	$|u(x)|<\infty$,$|\nabla u(x)|<\infty$,$ | a_{0}(x)|<\infty $
	and $u_{n}(x)\longrightarrow u(x)$, then by the growth condition \eqref{14},\eqref{12} and $||u_{n}||_{\infty}\leq c$ 
	$$D_{n}(x)\geq \frac{1}{(1+c)^{\theta(p-1)}}|\nabla u_{n}(x)|^{p-1}-c(x)\big( 1+|\nabla u_{n}(x) |+|\nabla u_{n}(x)|^{p-1} \big),$$
	where $c(x)$ is a constant depends on $x$ but does not depend on $n$, which schows thanks to \eqref{39}, that the sequence ${|\nabla u_{n}(x) |}$ is unformly bounded in $\mathbb{R}^{N}$, with respect to $n$, we argue  simililary as in Lemma 5 in \cite{12}, to obtain \eqref{qq1}. 
	
	We can now pass to the limit going back to the equation \eqref{33},
	to do this, let  $ \varphi \in W_{0}^{1,p}(\Omega)\cap L^{\infty}(\Omega) $. For every $ n> c_{\infty}$ one has 
	\begin{equation}\label{47}
	\int_{\Omega} a(x,u_{n},\nabla u_{n}) \nabla \varphi dx= \int_{\Omega}
	h_{n}(u_{n}) f_{n} \varphi dx, 
	\end{equation} 
	by \eqref{qq1}, we have  $\nabla u_{n}\longrightarrow  \nabla u $
	strongly in $(L^{p}(\Omega))^{N}$  and a.e in $\Omega$, so that  Vitali's Theorem implies that 
	$$ a(x,u_{n},\nabla u_{n}) \longrightarrow a(x,u,\nabla u) \,\,
	\mbox{strongly in} \,\, L^{p'}(\Omega)^{N}.$$ 
	Then, passing to the limit in \eqref{47}
	and using the result in the Step 1, we obtain 
	$$\int_{\Omega} a(x,u,\nabla u) \nabla \varphi dx =\int_{\Omega} f h(u)\varphi dx,$$
	for all $\varphi$ in $ W_{0}^{1,p}(\Omega)\cap L^{\infty}(\Omega) $, moreover, from \eqref{316} we have 
	$$ u \in W_{0}^{1,p}(\Omega) \cap L^{\infty}(\Omega).$$
\end{proof}
\begin{proof}[Proof of Theorems \ref{Athm2} and \ref{Athm3}.]
	Because the proofs of Theorems  \ref{Athm3} are similar to that of 
	Theorem  \ref{Athm2} , we restrict to the proof of Theorem \ref{Athm2} 
\end{proof}
\begin{proof}[Proof of Theorems \ref{Athm2}.]
	As consequence of Theorem $\ref{Athm6}$ there exist a subsequence, still
	indexed by $n,$ and a measurable function $u$ in 
	$ W_{0}^{1,p}(\Omega) \cap L^{r}(\Omega)$ such that $u_{n}$
	converges weakly to $u.$ Moreover, by Rellich Theorem we have 
	\begin{equation}\label{48}
	u_{n}\longrightarrow u \,\,\, a.e \,\,\,in\,\, \Omega.
	\end{equation} 
	Fix $k>0$, we will prove that 
	\begin{equation}\label{49}
	T_{k}(u_{n})\longrightarrow T_{k}(u) \,\,\, strongly \,\,\,in\,\,W_{0}^{1,p}(\Omega). 
	\end{equation} 
	By Theorem $\ref{Athm6}$, the sequence ${T_{k}(u_{n})} $ is bounded in 
	$W_{0}^{1,p}(\Omega)$. Therefore, by \eqref{48} we get 
	\begin{equation}\label{410}
	T_{k}(u)\rightharpoonup T_{k}(u) \,\,\, weakly  \,\,\,in\,\,W_{0}^{1,p}(\Omega). 
	\end{equation} 
	Using $T_{k}(u_{n})-T_{k}(u)$, which belongs to $
	W_{0}^{1,p}(\Omega) $, as test function in formulation \eqref{47}, we get 
	$$ \int_{\Omega} a(x, T_{n}(u_{n}),\nabla u_{n}) \nabla(T_{k}(u_{n})-T_{k}(u))dx  =\int_{\Omega} h_{n}(u_{n}) f_{n}(T_{k}(u_{n})-T_{k}(u)) dx.$$ 
	Thanks to \eqref{410} and \eqref{44}, we have 
	\begin{equation}\label{411}
	\lim_{n\rightarrow +\infty} \int_{\Omega} a(x, T_{n}(u_{n}),\nabla u_{n}) \nabla(T_{k}(u_{n})-T_{k}(u))dx=0.
	\end{equation}
	By the growth condition \eqref{14} and Theorem $\ref{Athm6}$, the sequence ${a(x,T_{n}(u_{n}),\nabla u_{n})}$ is bounded in $L^{p^{\prime}}(\Omega)^{N}$. Then, it converges weakly to some $l$ in $L^{p^{\prime}}(\Omega)^{N}$ and we obtain 
	\begin{equation}\label{412}
	\lim_{n\rightarrow +\infty} \int_{|u_{n}|\geq k} a(x,T_{n}(u_{n}),\nabla u_{n}) \nabla T_{k}(u)dx= \int_{|u|\geq k} l \nabla T_{k}(u) dx  =0.
	\end{equation}
	The continuity of the function $ a $, \eqref{48} and Vitali's theorem  allow us to have  
	$$ a(x,T_{n}(u_{n}),\nabla T_{k}(u)) \longrightarrow  a(x,u,\nabla T_{k}(u)) \,\,\,\,\,strongly\,\,  in \,\,\,L^{p^{'}}(\Omega)^{N}.$$
	Therefore, by Theorem $\ref{Athm6}$ and \eqref{410} we get 
	\begin{equation}\label{413}
	\lim_{n\rightarrow +\infty} \int_{\Omega} a(x,T_{n}(u_{n}),\nabla T_{k}(u_{n})) \nabla(  T_{k}(u_{n})-T_{k}(u))dx=0.
	\end{equation}
	On the other hand, we write for $n> k$
	\begin{align*}
	&\int_{\Omega} \left(a(x,T_{k}(u_{n}),\nabla T_{k}(u_{n}))-a(x,T_{k}(u_{n}),\nabla T_{k}(u))\right)
	\nabla(  T_{k}(u_{n})-T_{k}(u))dx\\
	&= \int_{\Omega} (a(x,T_{k}(u_{n}),\nabla T_{k}(u_{n})) \nabla(  T_{k}(u_{n})-T_{k}(u))dx\\
	& -  \int_{\Omega} (a(x,T_{k}(u_{n}),\nabla T_{k}(u)) \nabla(  T_{k}(u_{n})-T_{k}(u))dx \\
	&= \int_{|u_{n}|<k} (a(x,T_{n}(u_{n}),\nabla u_{n}) \nabla(  T_{k}(u_{n})-T_{k}(u))dx\\
	& -  \int_{\Omega} (a(x,T_{k}(u_{n}),\nabla T_{k}(u)) \nabla(  T_{k}(u_{n})-T_{k}(u))dx \\
	&= \int_{\Omega} (a(x,T_{n}(u_{n}),\nabla u_{n}) \nabla(  T_{k}(u_{n})-T_{k}(u))dx\\
	& -  \int_{|u_{n}|\geq k} (a(x,T_{n}(u_{n}),\nabla u_{n}) \nabla(  T_{k}(u_{n})-T_{k}(u))dx \\
	& - \int_{\Omega} (a(x,T_{k}(u_{n}),\nabla T_{k}(u)) \nabla(  T_{k}(u_{n})-T_{k}(u))dx.
	\end{align*}
	Observing that $\nabla T_{k}(u_{n})=0$ on the set ${|u_{n}|\geq k}$, we get 
	\begin{align*}
	&	\int_{\Omega} (a(x,T_{k}(u_{n}),\nabla T_{k}(u_{n}))-a(x,T_{k}(u_{n}),\nabla T_{k}(u)))
	\nabla(  T_{k}(u_{n})-T_{k}(u))dx \\
	&= \int_{\Omega} (a(x,T_{n}(u_{n}),\nabla u_{n}) \nabla(  T_{k}(u_{n})-T_{k}(u))dx\\
	&+ \int_{{|u_{n}|\geq k}} (a(x,T_{n}(u_{n}),\nabla u_{n}) \nabla( T_{k}(u))dx\\
	& -  \int_{\Omega} (a(x,T_{k}(u_{n}),\nabla T_{k}(u)) \nabla(  T_{k}(u_{n})-T_{k}(u))dx. 
	\end{align*}
	Thus,it follows from \eqref{411},\eqref{412} and \eqref{413} that 
	$$
	\int_{\Omega} (a(x,T_{k}(u_{n}),\nabla T_{k}(u_{n}))-a(x,T_{k}(u_{n}),\nabla T_{k}(u)))\nabla(T_{k}(u_{n})-T_{k}(u))dx\rightarrow0.$$
	when $n$ tends to $+\infty.$ By Lemma 5 of \cite{12}, we obtain \eqref{49}. The strong convergence \eqref{49} implies, for some subsequence still indexed by $n$, that 
	$$ \nabla u_{n} \longrightarrow \nabla u \,\,\, a.e\,\,\, .in  \,\,\, \Omega, $$
	which yields , since $ (a(x,T_{n}(u_{n}),\nabla u_{n}) $ is bounded in $L^{p'}(\Omega)^{N} $, that 
	$$ (a(x,T_{n}(u_{n}),\nabla u_{n}) \rightharpoonup  a(x,u,\nabla u)\,\,\, weakly \,\,\, in \,\,\,\,L^{p'}(\Omega)^{N}. $$
	Therefore ,passing to the limit in \eqref{47} we obtain \eqref{19}.
\end{proof}
\begin{proof}[Proof of Theorems \ref{Athm4}.]
	By Lemma \ref{lm1} and Lemma \ref{lm2}  the sequence ${u_{n}}$ is uniformly bounded in $ W_{0}^{1,p}(\Omega)$. Therefore we can obtain a solution passing to the limit, namely arguing exactly as in Theorem \ref{Athm2}.
\end{proof}
\bibliographystyle{unsrt}  


\end{document}